\documentclass[11pt, reqno]{amsart}
\usepackage{amsmath, amsfonts, amssymb, tikz-cd}
\usepackage{color}
\usepackage{graphicx}
\usepackage[bordercolor=gray!20,backgroundcolor=blue!10,linecolor=blue!50,textsize=footnotesize,textwidth=0.8in]{todonotes}
\usepackage{hyperref}

\newtheorem{thm}{Theorem}[section]
\theoremstyle{definition}

\newtheorem{cor}[thm]{Corollary}

\newtheorem{definition}[thm]{Definition}

\newtheorem{lem}[thm]{Lemma}
\newtheorem{prop}[thm]{Proposition}
\newtheorem{rem}[thm]{Remark}

\newcommand{\bordN}{\multicolumn{1}{c|}{-1}} 
\newcommand{\bordBN}{\multicolumn{1}{|c}{-1}} 
\newcommand{\bordP}{\multicolumn{1}{c|}{1}} 
\newcommand{\bordPP}{\multicolumn{1}{c|}{2}} 
\newcommand{\bordO}{\multicolumn{1}{c|}{0}} 
\newcommand{\bordBO}{\multicolumn{1}{|c}{0}} 
\newcommand{\bordB}{\multicolumn{1}{c|}{}} 
\newcommand{\bordBB}{\multicolumn{1}{|c}{}} 
\newcommand{\bordD}{\multicolumn{1}{|c}{\vdots}} 
\newcommand\bigzero{\makebox(0,0){\text{\huge0}}}

\newcommand{\R}{\mathbb{R}}

\newcommand{\Q}{\mathbb{Q}}

\begin{document}

\title{Comparing Bennequin-type Inequalities}

\author{Elaina Aceves} 
\address {Department of Mathematics, University of Iowa, Iowa City, IA 52242}
\email{elaina-aceves@uiowa.edu}

\author{Keiko Kawamuro}
\address {Department of Mathematics, University of Iowa, Iowa City, IA 52242}
\email{keiko-kawamuro@uiowa.edu}

\author{Linh Truong}
\address {Department of Mathematics, University of Michigan, Ann Arbor, MI 48103}
\email{tlinh@umich.edu}

\date{}

\maketitle

\begin{abstract} The slice-Bennequin inequality states an upper bound for the self-linking number of a knot in terms of its four-ball genus. The $s$-Bennequin and $\tau$-Bennequin inequalities provide upper bounds on the self-linking number of a knot in terms of the Rasmussen $s$ invariant and the Ozsv\'ath-Szab\'o $\tau$ invariant. We exhibit examples in which the difference between self-linking number and four-ball genus grows arbitrarily large, whereas the $s$-Bennequin inequality and the $\tau$-Bennequin inequality are both sharp.
\end{abstract}

\section{Introduction}

In the standard contact 3-space $(\R^3, \xi_{\text{std}})$, knots that are transverse to the contact planes can be viewed as braids around the $z$-axis. In this paper we will view transverse knots by their braid representations. We can describe the braid group $B_n$ with the standard generators $\sigma_1, \dots, \sigma_{n-1}$ and the following relations:
\begin{eqnarray*} && \sigma_i \sigma_j=\sigma_j \sigma_i \text{ for } |i-j|>1 \\
&& \sigma_i \sigma_{i+1}\sigma_i = \sigma_{i+1}\sigma_i\sigma_{i+1} \text{ for } i=1, \dots, n-2. 
\end{eqnarray*}

Let $K$ be a topological knot type in $S^3$. 
The self-linking number is an invariant of a transverse link. 
If a transverse knot is represented by a braid $\beta$ then the self-linking number can be computed using the following formula:
\[
sl(\widehat{\beta})=-n+a 
\]
where $\widehat\beta$ is the closure of $\beta$, $n$ is the braid index of $\beta$ and $a$ is the exponent sum of $\beta$ (or the algebraic crossing number of $\beta$).  
Given a topological knot type $K$ in $S^3$ we denote by $SL(K)$ the maximal value of the self linking numbers of transverse knot representatives and call it {\em the maximal self-linking number} of $K$. 
Bennequin \cite{Bennequin} showed $sl(\widehat\beta)\leq 2g_3(K) -1$ where $g_3(K)$ denotes the genus of the knot type $K$ that $\beta$ represents; thus, 
$$SL(K)\leq 2g_3(K) -1.$$

The quantities we examine in this paper include the maximal self-linking number $SL(K)$, the four ball genus $g_4(K)$, the Ozsv\'ath-Szab\'o concordance invariant $\tau(K)$ \cite{OS}, and the Rasmussen concordance invariant $s(K)$ \cite{R}. We also consider transverse invariants $\hat{\theta}(K)$ \cite{OST} from Heegaard Floer homology  and $\psi(K)$  \cite{P2} from Khovanov homology. 

For any knot type $K$, we have the following bounds on the self-linking number. 
\[
SL(K) \leq s(K)-1 \leq 2g_4(K)-1 \leq 2g_3(K)-1
\]
Rudolph \cite{Rudolph} proved $SL(K)\leq 2 g_4(K)-1$. 
Plamenevskaya \cite{P2}, Shumakovitch \cite{S} and Kawamura \cite{Kawamura} proved the first inequality $SL(K) \leq s(K)-1$. 
Rasmussen defined the $s$ invariant and proved that $s(K) \leq 2 g_4(K)$ in \cite{R} which gives us the second inequality.
In \cite{Pa}, Pardon extended the $s$ invariant from knots to links. Plamenevskaya's proof still applies with Pardon's definition, so we still have a bound for the self linking number. 

The concordance invariant $\tau(K)$ defined using Heegaard Floer homology \cite{OS} gives similar bounds \cite{OS, P3}:
\[
SL(K) \leq 2 \tau(K)-1 \leq 2g_4(K)-1 \leq 2g_3(K)-1
\]

\begin{definition}[\cite{HIK}] Let $K$ be a knot type in $S^3$. 
The {\em defect of the slice-Bennequin inequality} is defined as
$$\delta_4(K)=\frac{1}{2}(2g_4(K)-1-SL(K)).$$
\end{definition}

\begin{definition} Let $K$ be a knot in $S^3$. 
We define the {\em defect of the $s$-Bennequin inequality} as 
$$
\delta_s(K) = \frac{1}{2}(s(K)-1-SL(K)),
$$ 
and the {\em defect of the $\tau$-Bennequin inequality} as
$$
\delta_\tau(K) = \frac{1}{2}(2\tau(K)-1-SL(K)). 
$$ 
\end{definition}
Note that the defects $\delta_4$, $\delta_s$, and $\delta_\tau$ are always nonnegative. 

In our main result, we show that the defect $\delta_4(K)$ can be made arbitrarily large, while at the same time the defects $\delta_s(K)$ and $\delta_\tau(K)$ are both bounded. 
\begin{thm} \label{main theorem}
There exists a family of knots $K_n$, where $n = 1$, $2$, $\dots$, such that $\delta_4(K_n) = 2n$, whereas $\delta_s(K_n)=0$ and $\delta_\tau(K_n) =0.$ 
\end{thm} 

We give the first example of such an infinite sequence in the literature. 

Any knot satisfying Theorem \ref{main theorem} must be non-quasipositive. However, we will show in Section~\ref{section:thetapsi} that the nonquasipositive property of the knots $K_n$ is not detected by the Ozsv\'ath-Szab\'o-Thurston transverse invariant $\hat{\theta}(K)$ from knot Floer homology \cite{OST} and Plamenevskaya's $\psi(K)$ from Khovanov homology \cite{P2}. 

\begin{definition}
A braid $\beta \in B_n$ is {\em quasipositive} if it is a product of positive powers of some conjugates of the standard generators $\sigma_1, \dots, \sigma_{n-1}$. In other words, $\beta$ is quasipositive if it is conjugate to a braid word of the form 
\[
(w_1 \sigma_{i_1} w_1^{-1}) (w_2 \sigma_{i_2} w_2^{-1}) \cdots (w_k \sigma_{i_k}w_k^{-1})
\]
for some braid words $w_1, \dots, w_k$. 
A knot or link is then {\em quasipositive} if it can be represented by a quasipositive braid.  
\end{definition} 
We have the following result when $K$ is quasipositive.
\begin{prop}\label{prop0}
If $K$ is a quasipositive knot, then we have $$\delta_s(K)=\delta_\tau(K)=\delta_4(K)=0.$$ 
\end{prop} 
\begin{proof} Let $K$ be quasipositive.
Plamenevskaya \cite{P3} and Hedden \cite{H2} proved the equality $ SL(K) = 2 \tau(K)-1 $, and
 Plamenskaya \cite{P2} and Shumakovitch \cite{S} proved the equality
$SL(K)=s(K)-1.$ That the defect of the slice-Bennequin inequality of a quasipositive knot vanishes is well-known (see, for example, \cite[Proposition 1.10]{HIK}). 
\end{proof}

\subsection*{Acknowledgements}
EA is partially supported by the Ford Foundation. KK was partially supported by Simons Foundation Collaboration Grants for Mathematicians and NSF grant DMS-2005450. LT was partially supported by NSF grant DMS-200553. The authors would like to thank Gage Martin for useful conversation. 


\section{A sequence of nonquasipositive braids}\label{section:2} 
Throughout the rest of this paper, we focus on a particular sequence of braids and their knot closures. 
For each $n =1,2,\dots$, we define the $3$-stranded braid $\beta_n$ as
\[
\beta_n = (\sigma_1^{-1})^{2n+3} \sigma_2 (\sigma_1)^3 \sigma_2.
\]
The braid closure of $\beta_n$ is a knot denoted by $K_n = \widehat{\beta_n}$. The braid $\beta_n$ is shown in Figure \ref{fig1}. 
\begin{figure}
\includegraphics[height=3.5cm]{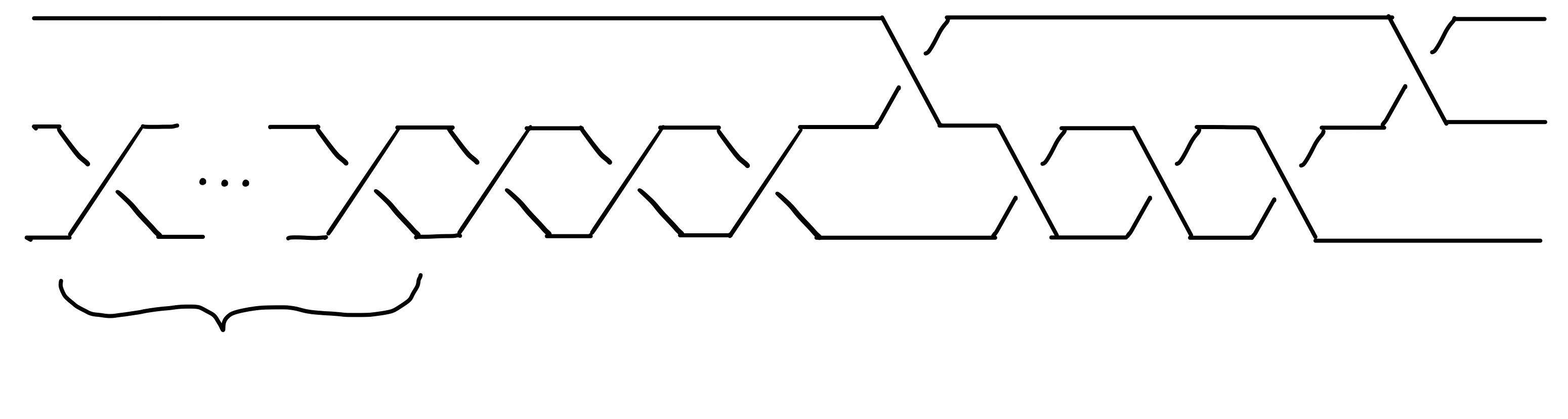}
\put(-325, 5){\fontsize{15}{11}$2n$}
\caption{The braid $\beta_n$. The braid closure $K_1=\widehat{\beta_1}$ is the knot $10_{125}$ and $K_2=\widehat{\beta_2}$ is the knot $12n235$. }
\label{fig1}
\end{figure}

\begin{thm}\label{Kn} For each $n = 1, 2, \dots$, let $K_n$ be the knot as constructed above. 
The defect of the slice Bennequin inequality for the knot $K_n$ is $\delta_4(K_n)=2n$. On the other hand, $\delta_s(K_n)=0$ and $\delta_\tau(K_n) = 0.$ 
\end{thm}

Theorem \ref{main theorem} from the Introduction follows from Theorem \ref{Kn}.

The proof of Theorem \ref{Kn} will rely on the signature bound on the four-ball genus: $\frac{1}{2}\sigma(K) \leq g_4(K)$. For the knots $K_n$, this signature bound will prove to be stronger than the $s$-invariant bound $\frac{1}{2}s(K)\leq g_4(K)$ and the $\tau$-invariant bound $\tau(K) \leq g_4(K)$.

\begin{proof}[Proof of Theorem~\ref{Kn}]
The result will follow from Corollary~\ref{cor1}, Proposition~\ref{prop3}, and Proposition~\ref{prop:tau}. 
\end{proof}

\subsection{Signature of $K_n$}

The goal of this section is to calculate the signature of the knots $K_n$. To do this, we will calculate the Seifert matrix for $K_n$ and prove that the signature of $K_n$ is $2n$. We begin with the case $n = 1$. 

To calculate the Seifert matrix for $K_1$, we will first find a surface $S$ with $K_1$ as its boundary. 
Consider the surface $S$ in Figure \ref{fig2} with the orientation induced by the orientation of the boundary $K_1$.  
\begin{figure}
\includegraphics[height=3.5cm]{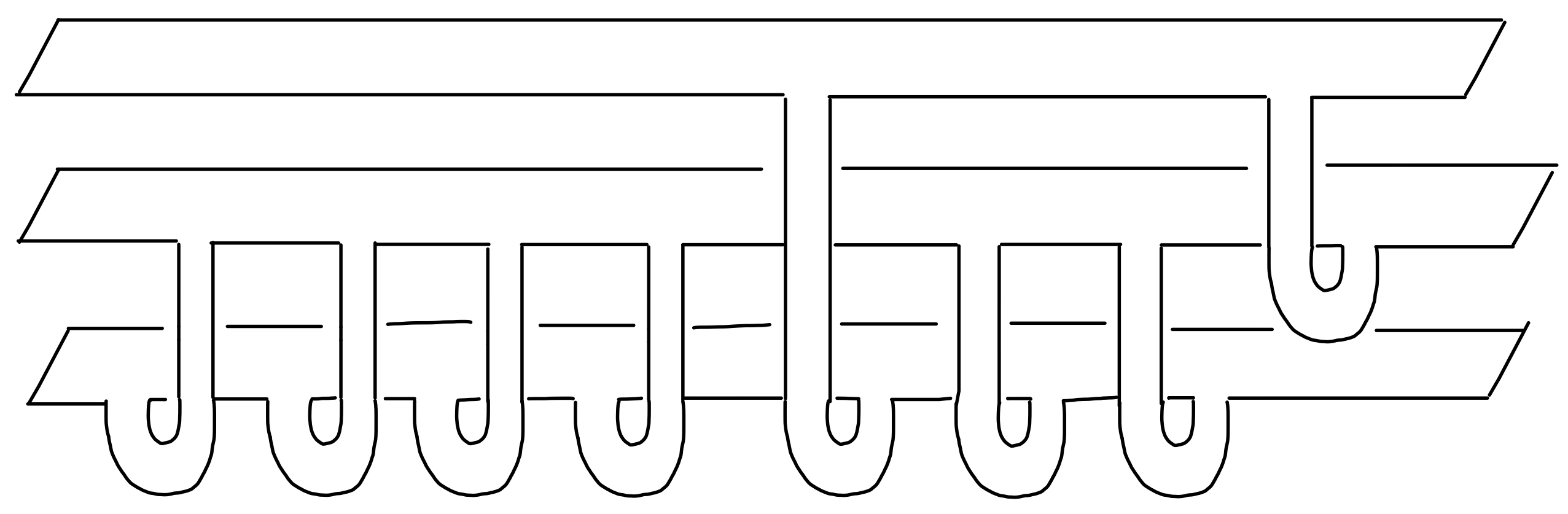}
\caption{Surface $S$ with $K_1$ as its boundary}
\label{fig2}
\end{figure}
The Euler characteristic of $S$ is 
\[
\chi(S)=3-8=-5.
\]
Since $\chi(S)=1-2g(S)$, the surface $S$ has genus 3. Next, we need to find basis curves $\gamma_1, \dots, \gamma_{6}$ that generate $H_1(S)$. 
Consider the oriented curves in Figure \ref{fig4}. 
Recall that the Seifert matrix has entries $lk(\gamma_i, \gamma_j^+)$ where $\gamma_j^+$ is the pushoff of $\gamma_j$ in the positive  normal direction of the surface. 
In Figure~\ref{fig4}, the pushoff $\gamma_1^+$ of $\gamma_1$ is shown. We have the following linking numbers:
\[
lk(\gamma_1, \gamma_1^+)=-2,  \quad lk(\gamma_2, \gamma_1^+)=0,\quad  lk(\gamma_3, \gamma_1^+)=-1,    
\]
\[
lk(\gamma_4, \gamma_1^+)=0,  \quad lk(\gamma_5, \gamma_1^+)=0,\quad  lk(\gamma_6, \gamma_1^+)=0.    
\]
We will denote the Seifert matrix as $V_1$. 
\[
V_1 = 
\begin{pmatrix}
-2 & 0 & -1 & 0 & 0 & 0 \\
-1 & -1 & 0 & 0 & 0 & 0 \\
0 & 1 & 0 & -1 & 0 & 0 \\
0 & 0 & 0 & 1 & -1 & 0 \\ 
0 & 0 & 0 & 0 & 1& -1 \\ 
0 & 0 & 0 & 0 & 0 &1
\end{pmatrix}
\]


\begin{figure}
\includegraphics[height=3.5cm]{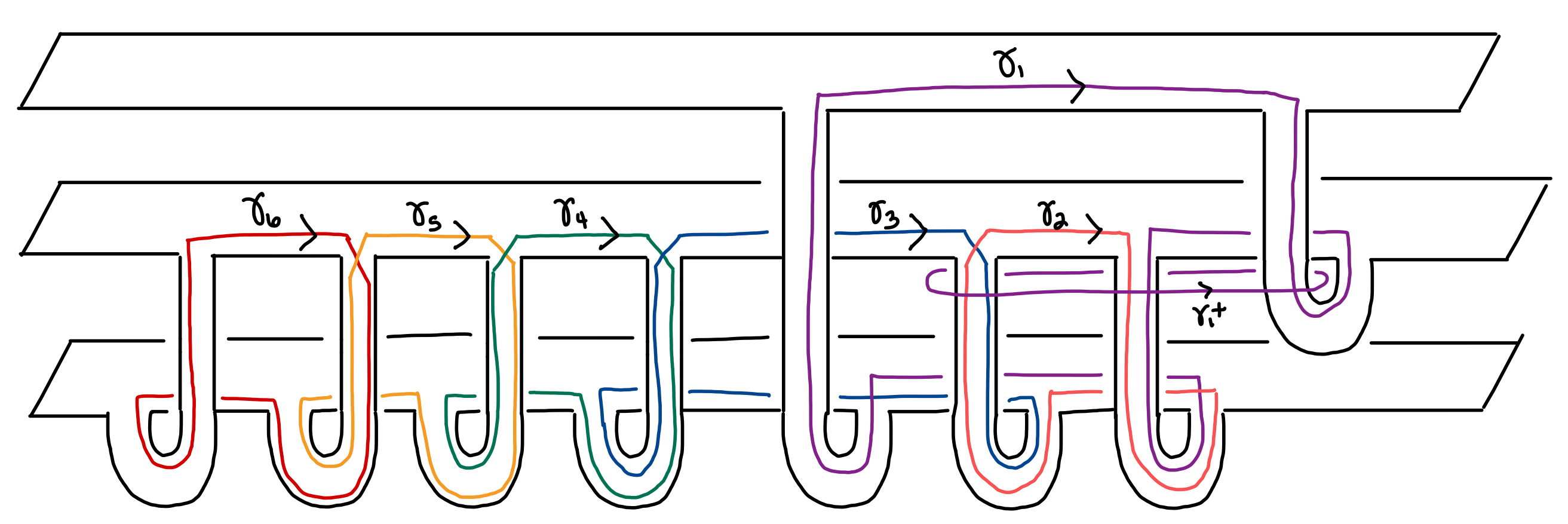}
\caption{The oriented curves $\gamma_1, \dots, \gamma_6$ in $S$ generate the homology group $H_1(S)$. The pushoff $\gamma_1^+$ links with other curves.}
\label{fig4}
\end{figure}

Now that we have a Seifert matrix associated with $K_1$, we will prove the following lemma about the signature $\sigma(K_1)$ of $K_1$. 

\begin{lem}\label{lem1}
The signature of $K_1$ is $\sigma(K_1)=2$.
\end{lem} 

\begin{proof}
In the previous discussion, we calculated the Seifert matrix $V_1$ for $K_1$. Recall that $\sigma(K_1)$ is the number of positive eigenvalues of $V_1+V_1^T$ minus the number of negative eigenvalues of $V_1+V_1^T$ where $V_1^T$ denotes the transpose of $V_1$.
The (symmetric) matrix $V_1+V_1^T$ is given below.
\[
V_1 +V_1^T= 
\begin{pmatrix}
-4 & -1 & -1 & 0 & 0 & 0 \\
-1 & -2 & 1 & 0 & 0 & 0 \\
-1 & 1 & 0 & -1 & 0 & 0 \\
0 & 0 & -1 & 2 & -1 & 0 \\ 
0 & 0 & 0 & -1 & 2 & -1 \\ 
0 & 0 & 0 & 0 & -1 & 2
\end{pmatrix}
\]

We will apply multiple row operations to $V_1+V_1^T$ to determine the sign of the eigenvalues.
We will denote the $i$th row in the matrix as $R_i$. 
In the following calculations, $R_i \to R_i \pm c R_j$ with $c \in \Q$ means that we replace the entries in $R_i$ with the  entries determined by the expression $R_i \pm c R_j$. 
In each step, the goal is to clear out all of the entries in a column except for the entry on the diagonal. 
We list the row operations performed at each step first and the resulting matrices after. 

\begin{eqnarray*}
\text{Step 1:} && R_2 \to R_2 - \frac{1}{4}R_1  \qquad \qquad \,\, \text{Step 5:} \quad \, R_1 \to R_1 +\frac{6}{5}R_5 \\ 
&&R_3 \to R_3 - \frac{1}{4}R_1 \qquad \qquad \qquad \qquad \,\,\, R_2 \to R_2 - \frac{7}{8}R_5 \\ 
\text{Step 2:} && R_1 \to R_1 - \frac{4}{7}R_2  \qquad \qquad \qquad \qquad \,\,\,R_3 \to R_3 + \frac{4}{5}R_5\\ 
&& R_3 \to R_3 + \frac{5}{7}R_2  \qquad \qquad \qquad \qquad \,\,\,R_4 \to R_4 + \frac{9}{10}R_5\\ 
\text{Step 3:} && R_1 \to R_1 + \frac{3}{2}R_3  \qquad \qquad \qquad \qquad \,\,\,R_6 \to R_6 + \frac{9}{10}R_5\\ 
&& R_2 \to R_2 - \frac{35}{32}R_3  \qquad \qquad \text{Step 6:} \quad R_1 \to R_1 +\frac{12}{11}R_6 \\ 
&& R_4 \to R_4 + \frac{7}{8}R_3 \qquad \qquad \qquad \qquad  \,\,\, R_2 \to R_2 -\frac{35}{44}R_6\\ 
\text{Step 4:} && R_1 \to R_1 + \frac{4}{3}R_4 \qquad \qquad \qquad \qquad  \,\,\, R_3 \to R_3 +\frac{8}{11}R_6\\ 
&&R_2 \to R_2 - \frac{35}{36}R_4 \qquad \qquad \qquad \qquad   R_4 \to R_4 +\frac{9}{11}R_6 \\ 
&&R_3 \to R_3+ \frac{8}{9}R_3  \qquad \qquad \qquad \qquad  \, \,\,R_5 \to R_5 + \frac{10}{11}R_6\\ 
&&R_5 \to R_5+ \frac{8}{9} R_5
\end{eqnarray*}

\newpage

\begin{eqnarray*}
V_1+V_1^T &\xrightarrow{\text{Step 1}}&
\begin{pmatrix}
-4 & -1 & -1 & 0 & 0 & 0 \\
  0 & -7/4 & 5/4 & 0 & 0 & 0 \\
0& 5/4 & 1/4& -1 & 0 & 0 \\
0 & 0 & -1 & 2 & -1 & 0 \\ 
0 & 0 & 0 & -1 & 2 & -1 \\ 
0 & 0 & 0 & 0 & -1 & 2
\end{pmatrix} \\
&\xrightarrow{\text{Step 2}}&
\begin{pmatrix}
-4 & 0 & -12/7 & 0 & 0 & 0 \\
 0 & -7/4 & 5/4 & 0 & 0 & 0 \\
 0 & 0&  8/7 & -1 & 0 & 0 \\
0 & 0 & -1 & 2 & -1 & 0 \\ 
0 & 0 & 0 & -1 & 2 & -1 \\ 
0 & 0 & 0 & 0 & -1 & 2
\end{pmatrix} \\
&\xrightarrow{\text{Step 3}}&
\begin{pmatrix}
-4 &  0 & 0 & -3/2 & 0 & 0 \\
 0 & -7/4 & 0&  35/32 & 0 & 0 \\
 0 & 0 &  8/7 & -1 & 0 & 0 \\
0 & 0 & 0&  9/8  & -1 & 0 \\ 
0 & 0 & 0 & -1 & 2 & -1 \\ 
0 & 0 & 0 & 0 & -1 & 2
\end{pmatrix}  \\
&\xrightarrow{\text{Step 4}}&
\begin{pmatrix}
-4 &  0 &  0 & 0&  -4/3 & 0 \\
 0 & -7/4 & 0 & 0& 35/36 & 0 \\
 0 & 0 &  8/7 & 0 & -8/9 & 0 \\
0 & 0 &  0 & 9/8  & -1 & 0 \\ 
0 & 0 & 0 & 0& 10/9 & -1 \\ 
0 & 0 & 0 & 0 & -1 & 2
\end{pmatrix} \\
&\xrightarrow{\text{Step 5}}&
\begin{pmatrix}
-4 &  0 &  0 &  0 &  0& -6/5 \\
 0 & -7/4 & 0 &  0  & 0& 7/8 \\
 0 & 0 &  8/7 & 0 & 0& -4/5 \\
0 & 0 &  0 & 9/8  & 0& -9/10 \\ 
0 & 0 & 0 & 0  & 10/9 & -1 \\ 
0 & 0 & 0 & 0 & 0 & 11/10
\end{pmatrix} \\
&\xrightarrow{\text{Step 6}}&
\begin{pmatrix}
-4 &  0 &  0 &  0 &  0  & 0 \\
 0 & -7/4 & 0 &  0  & 0 & 0\\
 0 & 0 &  8/7 & 0 &   0  & 0 \\
0 & 0 &  0 & 9/8  & 0 & 0 \\ 
0 & 0 & 0 & 0  & 10/9 & 0\\ 
0 & 0 & 0 & 0 & 0 &  11/10
\end{pmatrix}
\end{eqnarray*}

Since our row reduced matrix has two negative diagonal entries and four positive diagonal entries, there are two negative eigenvalues and four positive eigenvalues associated to our matrix. Thus, we have that $\sigma(K_1) = 4-2 =2$ and we have proved the result.
\end{proof}

We can generalize the construction of the Seifert surface for $K_1$ to create for each $n \geq 2$ a surface $T_n$ with boundary $K_n$ as seen in Figure \ref{fig5}. 
The box labelled $B(n)$  in Figure \ref{fig5} represents $2n-3$ negative bands between the bottom two disks. The oriented curves $\gamma_1, \dots, \gamma_{2n+4}$ generate $H_1(T_n)$. The curves $\gamma_6, \dots, \gamma_{2n-2}$, which encircle adjacent bands (similar to $\gamma_4, \gamma_5$ and $\gamma_6$ from Figure \ref{fig4}), as well as the other half of the curves $\gamma_5$ and $\gamma_{2n+3}$ are not drawn but are also represented by the box. All of the curves are oriented in the same direction, namely oriented clockwise.
The Seifert matrix $V_n$ associated to the knot $K_n$ is of size $2n+4$ by $2n+4$ and given below. 

\begin{figure}
\includegraphics[height=3.5cm]{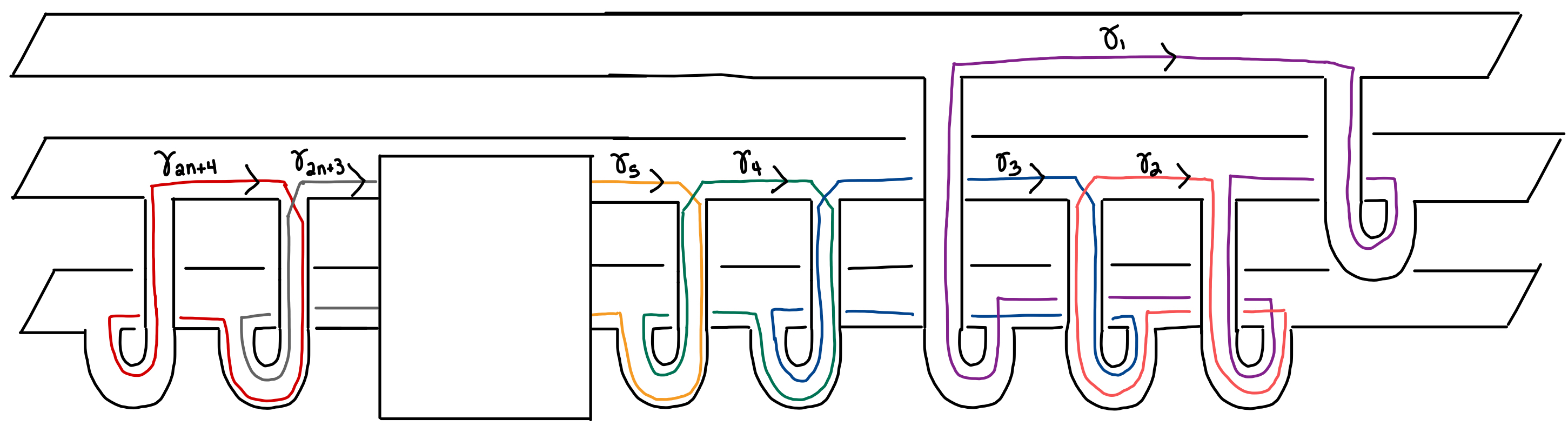}
\put(-265, 30){\fontsize{15}{11}$B(n)$}
\caption{The surface $T_n$ has boundary $K_n$. The oriented curves $\gamma_1, \dots, \gamma_{2n+4}$ generate $H_1(T_n)$.}
\label{fig5}
\end{figure}

\begin{eqnarray*}
V_n = 
\left(
\begin{array}{ccccccccc}
-2 & 0  &  -1 &\bordO &  &  & &   \\
-1 & -1 & 0 & \bordO & &  & \bigzero &\\ 
0 & 1 & 0 & \bordN  & & &  &  \\ \cline{5-5}
0 & 0 & 0 & \bordP & \bordN &   \\ \cline{1-1} \cline{2-2} \cline{3-3} \cline{4-4} \cline{6-6}
 & & & \bordB & \bordP & \bordN &  &  \\ \cline{5-5}
 & & & & \bordB & \bordP &  \\ \cline{6-6}
 & & \bigzero& & & & \ddots \\ \cline{8-8}
  & & & & & & \bordB & -1  \\ 
   & & & & & & \bordB & 1\\ 
  \end{array}\right)
\end{eqnarray*}

Moreover, we can construct the matrix $V_n+V_n^T$ associated to $K_n$, again of size $2n+4$ by $2n+4$.

\begin{eqnarray*}
V_n+V_n^T = 
\left(
\begin{array}{ccccccccc}
-4 & -1  &  \bordN & &  &  & &   \\
-1 & -2 & \bordP & & &  & \bigzero &\\ \cline{4-4}
-1 & 1 & \bordO &  \bordN & & &  &  \\ \cline{1-1} \cline{2-2} \cline{3-3} \cline{5-5}
 &  \bordB & -1 & \bordPP & \bordN &   \\  \cline{3-3} \cline{4-4} 
 & &  \bordB & -1 & \bordPP &  &  \\ \cline{4-4} \cline{5-5}
 & \bigzero& & &  & \ddots &  \\ \cline{8-8}
  & &  & & & & \bordB & -1  \\ \cline{7-7}
   & & & & &  \bordB &-1 & 2\\ 
  \end{array}\right)
\end{eqnarray*}

We inductively define the square matrices $M_n$ for $n \geq 1$ as follows. Let $M_1$ denote $M_1 = V_1 + V_1^T$. Let $M_n$ be obtained from $M_{n-1}$ as in the following diagram. 

\begin{eqnarray*}
M_1 = 
\left(
\begin{array}{cccccc}
-4 & -1  &  \bordN & 0 & 0 & 0    \\
-1 & -2 & \bordP & 0 & 0 & 0 \\ \cline{4-4}
-1 & 1 & \bordO &  \bordN & 0 & 0 \\ \cline{1-1} \cline{2-2} \cline{3-3} \cline{5-5}
0 & \bordO & -1 & \bordPP & \bordN &0  \\  \cline{3-3} \cline{4-4} \cline{6-6}
0 & 0 &  \bordO & -1 & \bordPP & -1  \\ \cline{4-4} \cline{5-5}
0 & 0 & 0 & \bordO& -1& 2   \\
  \end{array}\right) \quad 
M_n = 
\left(
\begin{array}{ccccc}
& & & \bordB & 0 \\
&  & & \bordB & \vdots \\
& & & \bordB & 0 \\\cline{5-5}
& & & \bordB & -1 \\ \cline{1-1} \cline{2-2} \cline{3-3} \cline{4-4}
0 & \cdots & \bordO & -1 & 2 \\
  \end{array}\right)
  \put(-105, 7){\fontsize{25}{11}$M_{n-1}$}
\end{eqnarray*}

Observe that $M_{2n-1}=V_n+V_n^T$. Note that the matrix $M_n$ is an $(n + 5) \times (n+5)$ matrix.
To calculate the signature of the knot $K_n$, we calculate the signature of the matrix $M_{2n-1}$.

\begin{lem}\label{lem2}
We can reduce $M_n$ to the matrix $\widetilde{M}_n$ using only row operations in the first $n+4$ rows where an asterisk designates that the entry could be any rational number.
\end{lem} 

\begin{eqnarray*}
\widetilde{M}_n = 
\left(
\begin{array}{ccccccc}
-4 & 0 & \bordBB& & & \bordB& *  \\
0 & -7/4& \bordBB & & \bigzero& \bordB & * \\ \cline{1-2}
& & 8/7 & & & \bordB & * \\
& & & 9/8& &  & \bordD \\ 
& \bigzero&  & & \ddots & \bordB & * \\ \cline{7-7}
& & & & & \frac{n+9}{n+8}& \bordBN \\ \cline{6-6}
& & & & & \bordBN & 2\\
  \end{array}\right)
\end{eqnarray*}

\begin{proof}
We will prove this by induction on $n$. 

As our base case, we have already shown that $M_1$ can be reduced to $\widetilde{M}_1$ using row operations in the first five rows by following Steps 1-4 and the first four row operations from Step 5 in Lemma \ref{lem1}. Hence the base case is satisfied.  
\begin{eqnarray*}
\widetilde{M}_1 = 
\left(
\begin{array}{cccccc}
-4 & 0 & \bordBB& \bigzero&  \bordB& -6/5  \\
0 & -7/4& \bordBB & & \bordB & 7/8 \\ \cline{1-2}
& & 8/7 & &  \bordB & -4/5 \\
& & & 9/8&  \bordB &-9/10 \\ \cline{6-6}
& \bigzero & &  &10/9 & \bordBN \\  \cline{5-5}
& & & & \bordBN& 2 \\ 
  \end{array}\right)
\end{eqnarray*}

As our inductive hypothesis, assume we can reduce $M_n$ to $\widetilde{M}_n$ for $n \geq 1$ using row operations in the first $n+4$ rows. Recall $M_{n+1}$ contains $M_n$ as a submatrix. 
By the inductive hypothesis, we can row reduce the embedded matrix $M_n$ using row operations in the first $n+4$ rows of $M_{n+1}$.  
Since the last column of $M_{n+1}$ has zeros in the first $n+4$ entries, the last column is unaffected by these row operations. After performing the row operations, we obtain the resulting matrix, which we denote by $M_{n+1}'$, shown below. 

\begin{eqnarray*}
M_{n+1}' = 
\left(
\begin{array}{cccccccc}
-4 &  & & & & \bordB& * & \bordBO  \\
 & -7/4&& & \bigzero& \bordB & * & \bordBO\\ 
& & 8/7 & & & \bordB & * & \bordBO \\
& & & 9/8& & \bordB & \bordD & \bordBO \\ 
& \bigzero&  & & \ddots & \bordB & * & \bordD \\ \cline{7-7}
& & & & & \frac{n+9}{n+8}& \bordBN & \bordBO \\ \cline{6-6} \cline{8-8}
& & & & & \bordBN & 2 & \bordBN \\ \cline{1-7}
0& 0& 0& 0 & \cdots & \bordO & -1 & 2\\
\end{array}\right)
\end{eqnarray*}

We now perform multiple row operations. 
In Step A, we perform only one row operation in the second to last row, specifically $R_{n+5} \to R_{n+5}+\frac{n+8}{n+9} R_{n+4}$. 
In Step B, we use row operations to force the $(n+5)$th column to have zeros in the first $n+4$ many entries. 
Notice that this will introduce values in the first $n+4$ many entries in the last column and we need to use row operations only in the first $n+5$ many rows.  
The resulting matrix is $\widetilde{M}_{n+1}$, and the result is proved. 

\begin{eqnarray*}
M_{n+1}' &\xrightarrow{\text{Step A}}& \left(
\begin{array}{cccccccc}
-4 &  & & & & \bordB& * & \bordBO  \\
 & -7/4&& & \bigzero& \bordB & * & \bordBO\\ 
& & 8/7 & & & \bordB & * & \bordBO \\
& & & 9/8& &  & \bordD & \bordBO \\ 
& \bigzero&  & & \ddots & \bordB & * & \bordD \\ \cline{7-7}
& & & & & \frac{n+9}{n+8}& -1 & \bordBO \\ \cline{8-8}
& & & & &  0 & \frac{n+10}{n+9} & \bordBN \\ \cline{1-7}
0& 0& 0& 0 & \cdots & \bordO & -1 & 2\\
\end{array}\right) \\
&\xrightarrow{\text{Step B}}&   \left(
\begin{array}{cccccccc}
-4 &  & & & & &  0 & *   \\
 & -7/4&& & \bigzero&  &  0& * \\ 
& & 8/7 & & & &  0 & * \\
& & & 9/8& &  &  0 & *  \\ 
& \bigzero&  & & \ddots &  & \vdots & \vdots  \\ 
& & & & & \frac{n+9}{n+8}&  0 & *  \\ \cline{8-8}
& & & & &  & \frac{n+10}{n+9} & \bordBN \\ \cline{1-7}
0& 0& 0& 0 & \cdots & \bordO & -1 & 2\\
\end{array}\right) 
\end{eqnarray*}
\end{proof}

Now that we have proved Lemma \ref{lem1}, we are prepared to calculate the signature of each matrix $M_n$. 

\begin{lem}\label{lem3}
For $n=1,2, \dots$, the matrix $M_n$ has signature $\sigma(M_n) = n+1$. 
\end{lem}

\begin{proof}
Consider the matrix $M_n$. 
By Lemma \ref{lem2}, we can row reduce $M_n$ to $\widetilde{M}_n$ using only row operations in the first $n+4$ rows. 
Performing Steps A and B from Lemma \ref{lem2} forces all entries not on the diagonal of $\widetilde{M}_n$ to be zero. 
Counting the number of positive and negative values along the diagonal, we can conclude that $
\sigma(M_n) = (n+3)-2=n+1$. 

\begin{eqnarray*}
\widetilde{M}_{n} &\xrightarrow{\text{Step A}}& \left(
\begin{array}{cccccccc}
-4 &  & & && &  *  \\
 & -7/4&& && \bigzero & * \\ 
& & 8/7 & & && *  \\
& & & 9/8& && \vdots  \\ 
& \bigzero&  &&  \ddots & & *  \\
& & & & &\frac{n+9}{n+8}& -1  \\
& & & &  &0 & \frac{n+10}{n+9}  \\ 
\end{array}\right) \\
&\xrightarrow{\text{Step B}}&   \left(
\begin{array}{cccccccc}
-4 &  & & & & &0  \\
 & -7/4&& & \bigzero& &  0 \\ 
& & 8/7 & & & & 0 \\
& & & 9/8& &  & 0 \\ 
& \bigzero&  & & \ddots &  &  \vdots   \\ 
& & & & & \frac{n+9}{n+8}& 0  \\ 
& & & & &  & \frac{n+10}{n+9} \\
\end{array}\right) 
\end{eqnarray*}
\end{proof}

We are finally ready to calculate the signature of $K_n$. 

\begin{prop}\label{prop1}
For $n=1,2, \dots$, the knot $K_n$ has signature $\sigma(K_n) =2n$. 
\end{prop}

\begin{proof}
Recall that the Seifert matrix $V_n+V_n^T$ associated to each knot $K_n$ is the matrix $M_{2n-1}$.
By Lemma \ref{lem3}, we conclude that 
$\sigma(K_n) = 2n.$
\end{proof}


\subsection{Four-ball genus of $K_n$}

The goal of this section is to calculate the four-ball genus $g_4(K_n)$. We will use Murasugi's \cite[Theorem 9.1]{M} lower bound on $g_4(K_n)$ in terms of the signature of $K_n$ and directly construct a sequence of surfaces with boundary $K_n$.

\begin{prop}\label{prop:g4}
For each $n=1,2, \dots$, the knot $K_n$ has four-ball genus $g_4(K_n)=n$. 
\end{prop}

\begin{proof}
We construct a surface $S_n$ in $B^4$ with $g(S_n)=n$ and $K_n$ as its boundary. 
We illustrate the procedure for $n=1$. 
Begin with $\beta_1$ and perform braid isotopy until you arrive at Figure \ref{fig6}. 

\begin{figure}
\includegraphics[height=4.5cm]{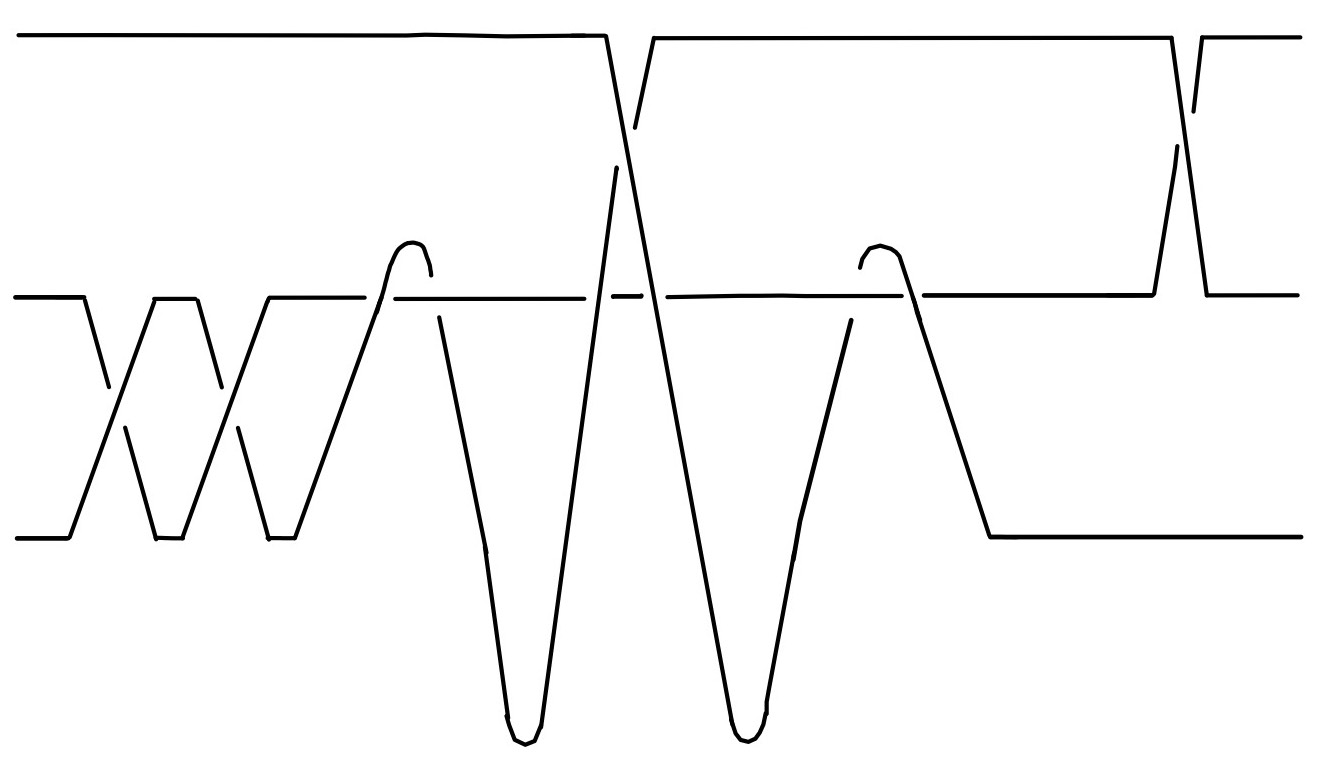}
\caption{$K_1$ after isotopy}
\label{fig6}
\end{figure}  

We create $S_1$ with $K_1$ as its boundary as seen in Figure \ref{fig7}. 
Notice that we introduced bands at each standard crossing and the remaining crossings contribute to one band with two ribbon intersections which are in green in Figure \ref{fig7}.
To better understand this band with ribbon intersections, we have Figures \ref{fig8} and \ref{fig9}.
In Figure \ref{fig8}, we have colored the band to illustrate how it wraps around and through the three horizontal parallel disks. 
The top of $S_1$ is highlighted in solid pink while the bottom of $S_1$ is in dashed blue.
The two ribbon intersections are still highlighted in green. 
On the left of Figure \ref{fig9}, we have colored the boundary of the three disks and the large band that makes up $S_1$ and on the right of Figure \ref{fig9}, we have $S_1$ as viewed from the right hand side with this new coloring. 
Notice that we are ignoring the three bands from the standard crossings in Figure \ref{fig9} and they remain uncolored.
The three disks, that are colored pink, yellow, and dark blue (from top to bottom), when viewed from the right hand side look like line segments. 
The band is highlighted in red and begins at the black dot on the pink disk and ends at the black dot on the dark blue disk.
Notice how the band wraps around and through the disks, creating the two ribbon intersections, as the band begins at the black dot on the pink disk, passes through the dark blue disk and then the yellow disk before stopping at the black dot on the dark blue disk.
Finally, push a disk neighborhood of the ribbon intersections, which is highlighted in blue in Figure \ref{fig7}, into the 4-ball. This will resolve the ribbon intersection and the resulting surface, which we call $S_1$, is properly embedded in $B^4$.

We calculate the Euler characteristic of $S_1$ using the fact that there are three disks and four bands. 
\[
\chi(S_1) = 3 - 4 = -1. 
\] 
Since $\chi(S)=1-2 g(S)$ for knots, we have that $g(S_1) = 1$. 

\begin{figure}
\includegraphics[height=4.5cm]{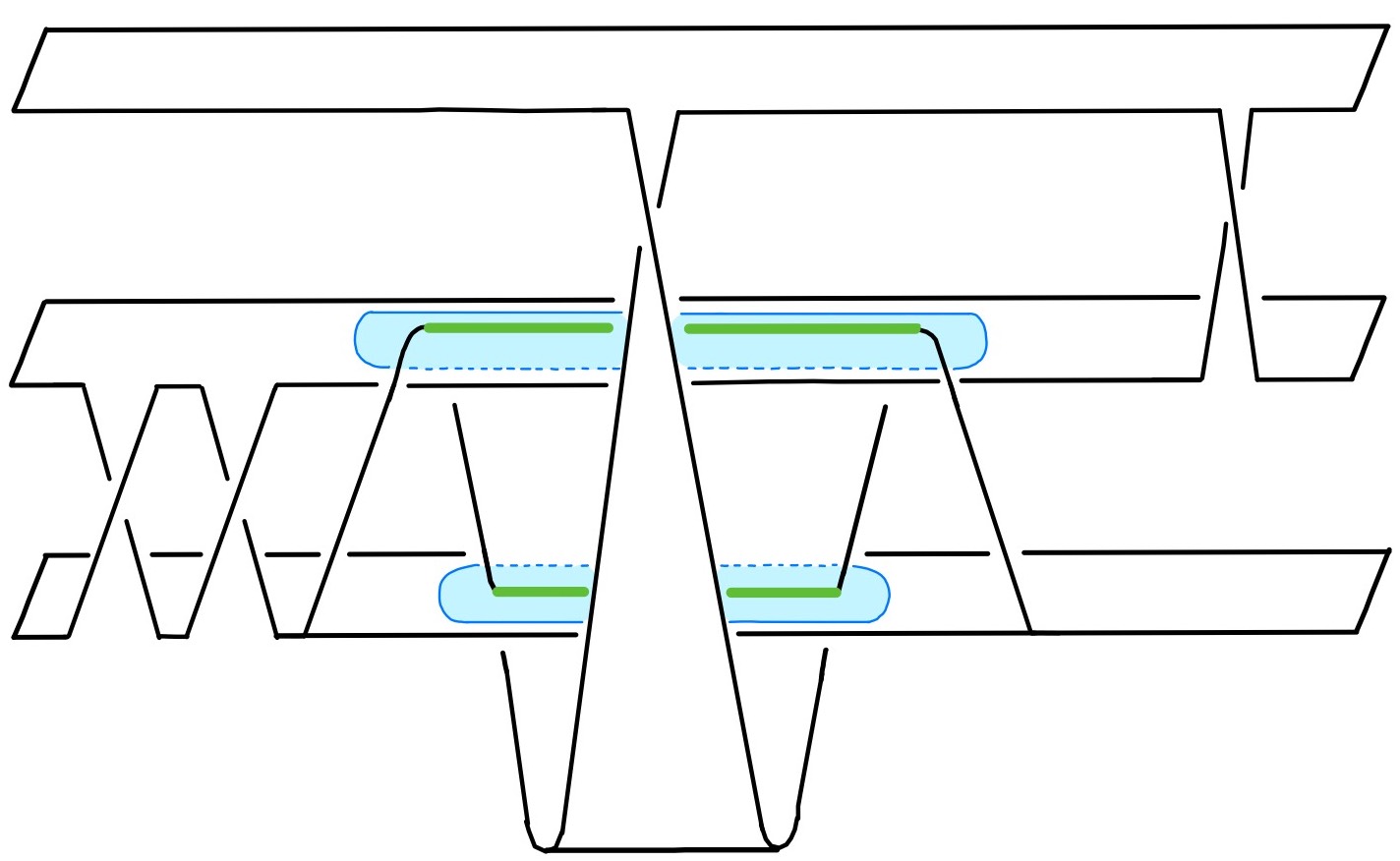}
\caption{$S_1$ with ribbon intersections}
\label{fig7}
\end{figure}  

\begin{figure}
\includegraphics[height=4.5cm]{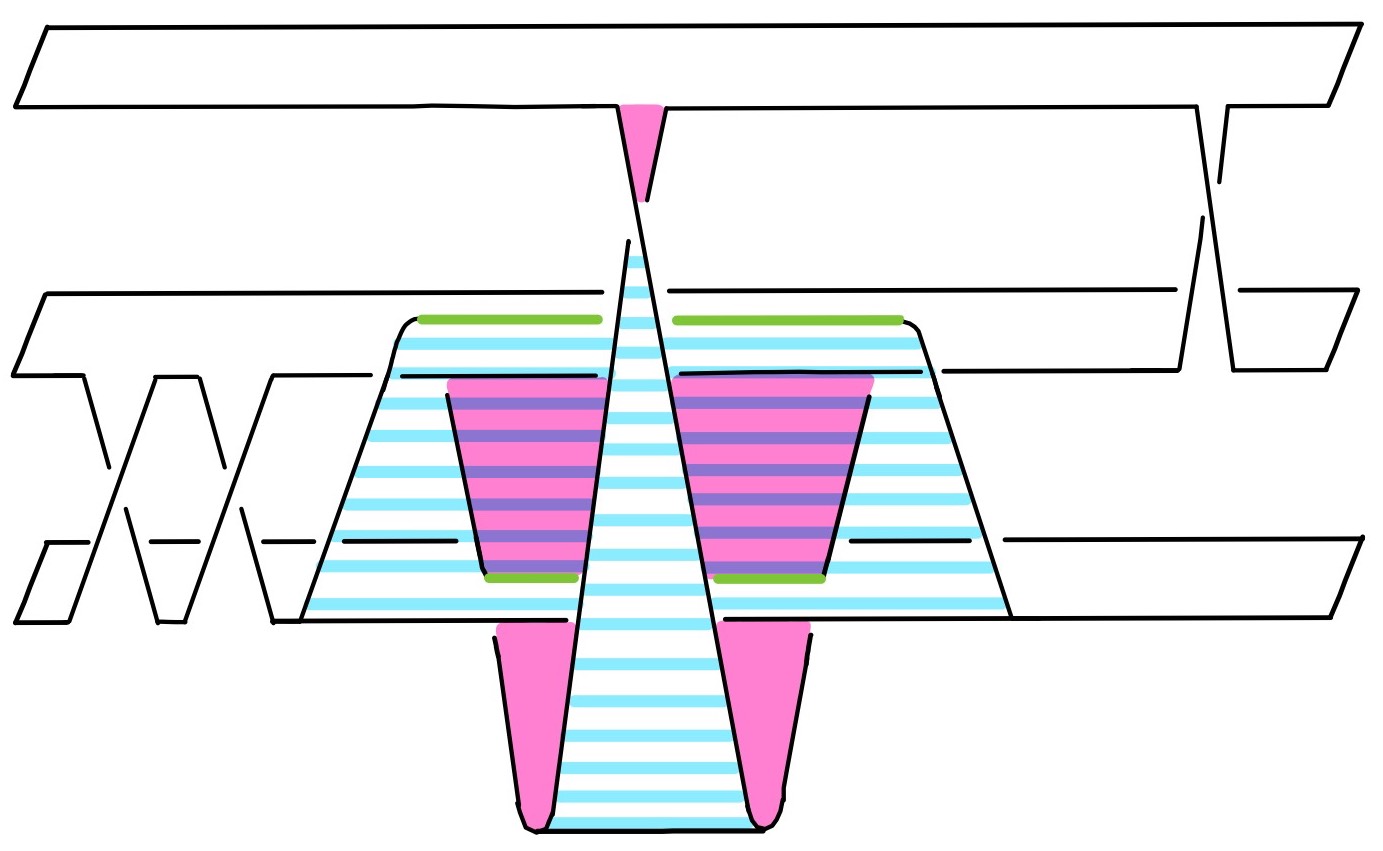}
\caption{$S_1$ with colored band}
\label{fig8}
\end{figure}  

\begin{figure}
\includegraphics[height=4.5cm]{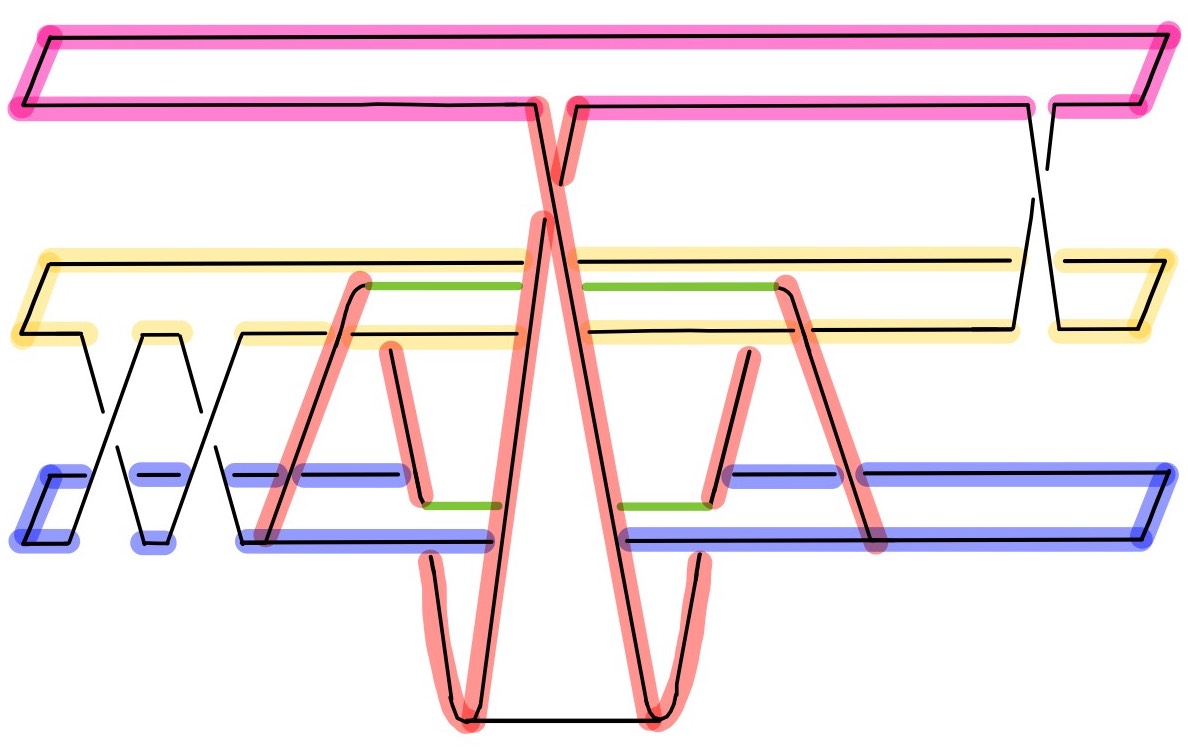} \quad 
\includegraphics[height=3.5cm]{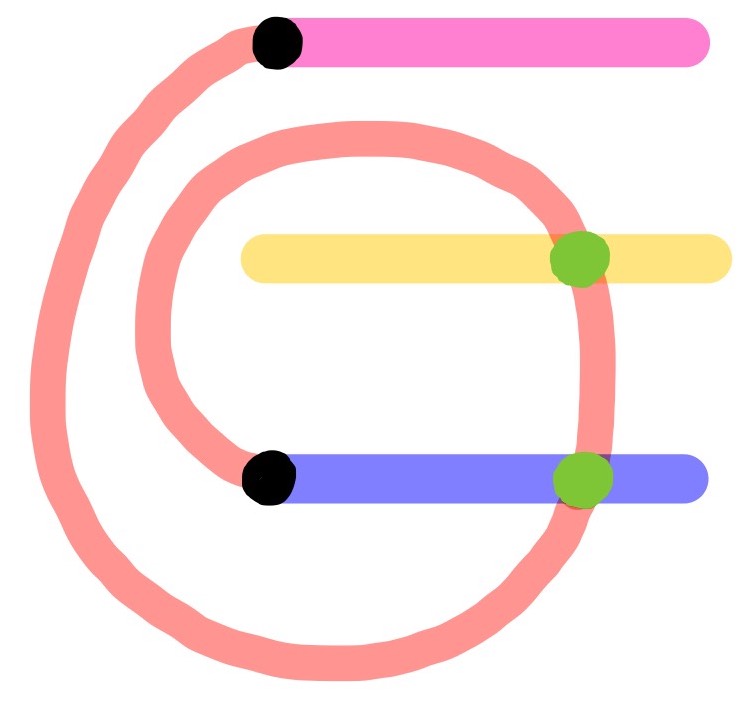}
\caption{Band in $S_1$}
\label{fig9}
\end{figure}  

We can create a surface $S_n$ with $K_n$ as its boundary by simply having $2n$ many negative bands instead of the two negative bands we have on the left in $S_1$. Again the ribbon intersections are resolved and $S_n$ is a properly embedded smooth surface in $B^4$.

Thus, $S_n$ would have a total of $2n+2$ many bands comprising of the $2n$ negative bands on the left of the surface, the large band that has two ribbon intersections, and one positive band on the right of the surface. 
We calculate the Euler characteristic of the surface $S_n$ 
\[
\chi(S_n)=3- (2n+2)=1-2n
\]
and we have that $g(S_n) = n$.
Hence $g_4(K_n) \leq n$.  

K. Murasugi proved that $\frac{1}{2} |\sigma(K)| \leq g_4(K)$ in \cite[Theorem 9.1]{M}.
By Proposition \ref{prop1}, we have that $\frac{1}{2} (2n) \leq g_4(K_n)$. 
Hence  $g_4(K_n) =n$. 
\end{proof}

\begin{cor}\label{cor1}
For each $n=1,2,\dots$, the defect $\delta_4(K_n)=2n$. In particular, $K_n$ is nonquasipositive.
\end{cor}

\begin{proof}
We compute the self-linking number of braids $\beta_n$ in our sequence and obtain:
\[
sl(\widehat\beta_n)=-3+5-(2n+3)=-2n-1. 
\] 
By the generalized Jones conjecture \cite{DP, L}, the maximal self linking number can be realized at the minimal braid index. As the braid index of $K_n$ is $3$ and $\beta_n$ is a 3-braid, we obtain $$SL(K_n)=sl(\widehat\beta_n)=-2n-1.$$ 

By Proposition~\ref{prop:g4} we have $g_4(K_n)=n$. We compute the defect 
$$\delta_4(K_n)= \frac{1}{2}(2 g_4(K_n) -1-SL(K_n))=2n.$$ 
By Proposition~\ref{prop0} we conclude that $K_n$ is nonquasipositive. 
\end{proof}


\subsection{The $s$ invariant of $K_n$}
The goal of this section is to calculate the $s$ invariant of $K_n$. We will determine which braid word $\beta_n$ is conjugate to under Murasugi's classification of 3-braids in \cite{M2}.
Then we use this new braid word to calculate the $s$ invariant for each $K_n$. 

\begin{lem}\label{lem4} 
For $n=1,2, \dots$, $\beta_n$ is conjugate to the braid $$A_n =(\sigma_1 \sigma_2)^3 \sigma_1 (\sigma_2^{-1})^{2n+5}$$
that belongs to the first type in the Murasugi classification of 3-braids \cite{M2}.
\end{lem} 

\begin{proof}
We begin by examining the braid $A_n=(\sigma_1 \sigma_2)^3 \sigma_1 (\sigma_2^{-1})^{2n+5}$ depicted at the top of Figure \ref{fig10}. 
Note that the box labeled $T$ contains $2n+2$ many negative twists, or $2n+2$ many $\sigma_2^{-1}$'s throughout the figure.
Using conjugation, we are able to move the negative crossing highlighted in blue along $\sigma_1$ and $\sigma_2$ to cancel with a $\sigma_1$. 
Similarly, we can move the negative crossing highlighted in pink underneath $\sigma_1$ and $\sigma_2$ to cancel with another $\sigma_1$. 
The last braid is conjugate to $\beta_n$ and we are done. 
\end{proof}

\begin{figure}
\begin{eqnarray*} 
& \raisebox{20pt}{ $A_n =$} & \includegraphics[height=1.5cm]{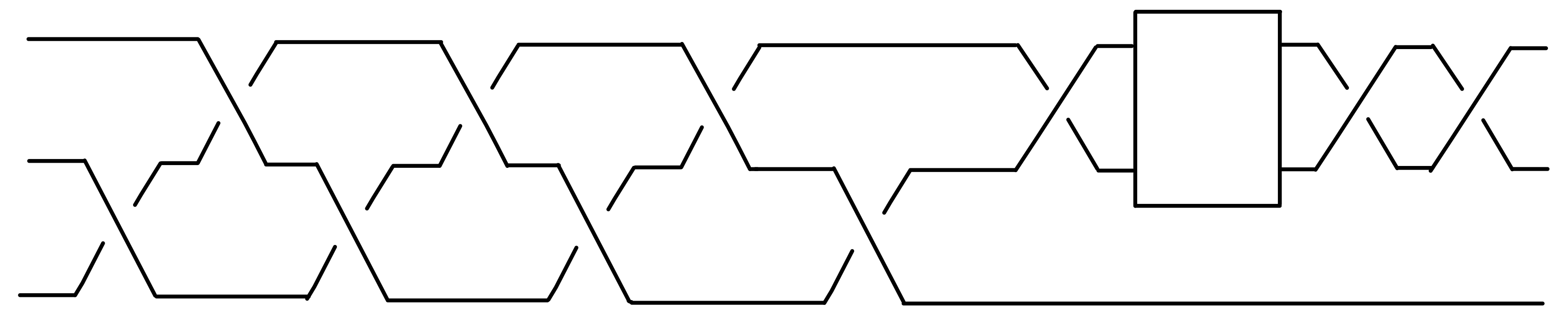} \put(-53, 25){\fontsize{15}{11}$T$} \\ 
&\raisebox{20pt}{ $\sim$} & \includegraphics[height=1.5cm]{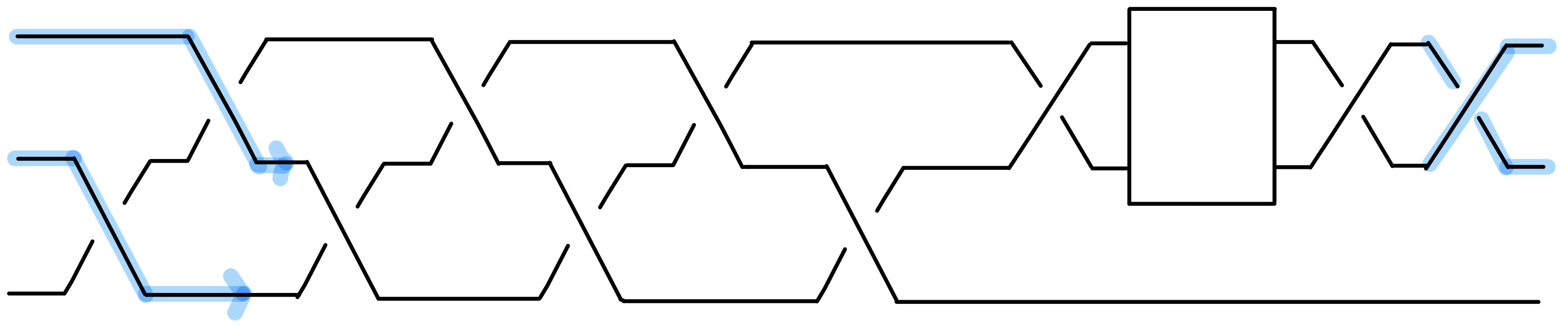}  \put(-53, 25){\fontsize{15}{11}$T$} \\
&\raisebox{20pt}{ $\sim$}& \includegraphics[height=1.5cm]{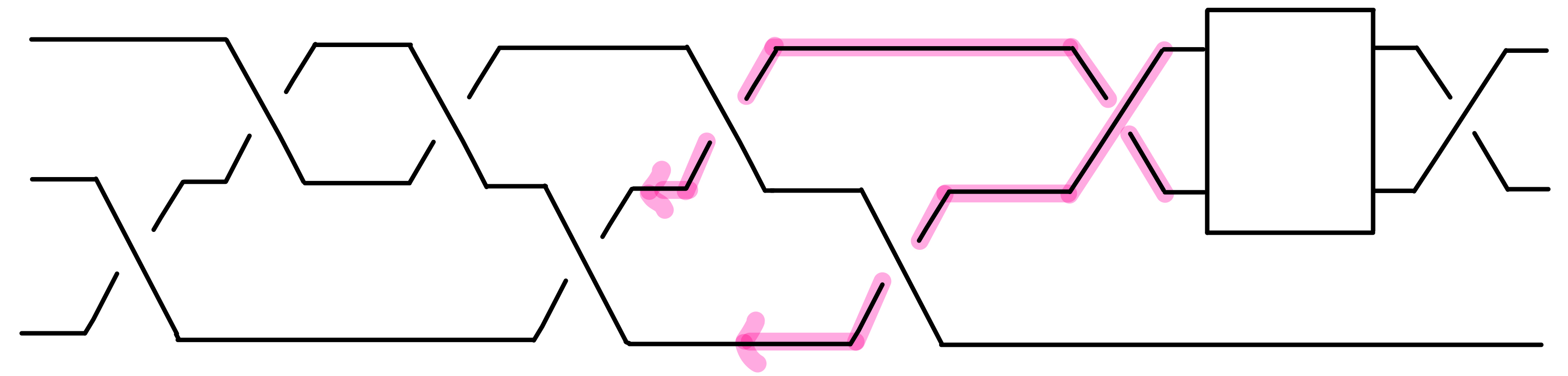}  \put(-36, 25){\fontsize{15}{11}$T$}\\
&\raisebox{20pt}{ $\sim$}& \includegraphics[height=1.5cm]{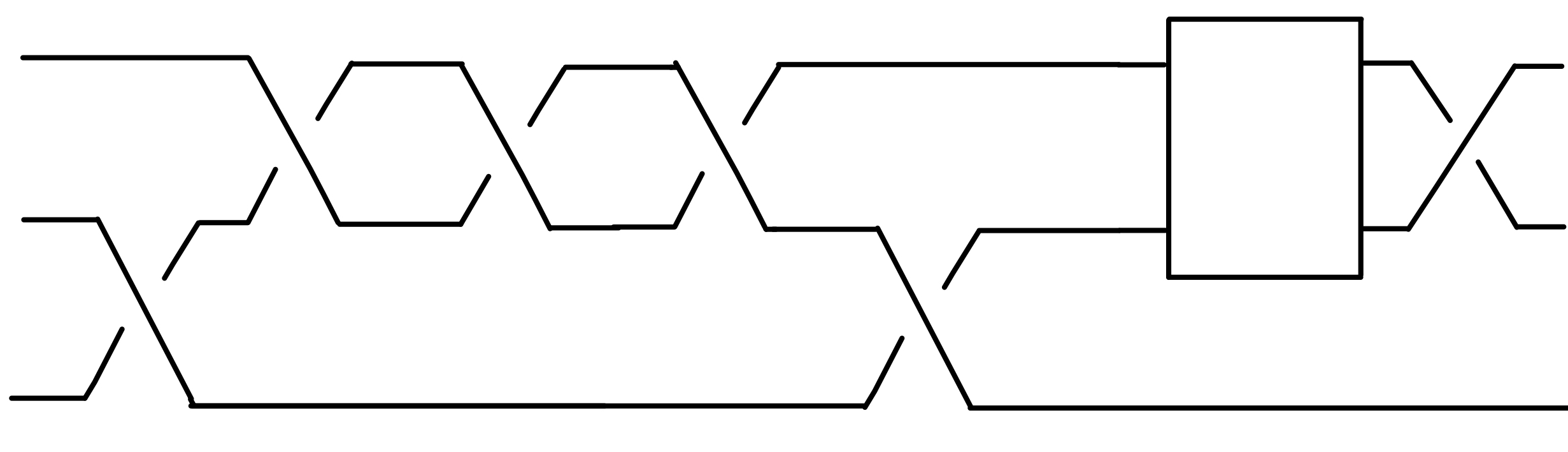}  \put(-34, 25){\fontsize{15}{11}$T$}\quad \raisebox{20pt}{ $\sim K_n$}
\end{eqnarray*}
\caption{Showing $K_n$ is conjugate to $A_n$}
\label{fig10}
\end{figure}

\begin{prop}\label{prop3}
For $n=1,2,\dots$, $s(K_n)=-2n$; thus, $\delta_s(K_n)=0$. 
\end{prop}

\begin{proof}
By Lemma \ref{lem4}, we know that $A_n$ is of Type 1 according to Murasugi's classification of 3-braids with $d=1$ and $a_1=2n+5$ \cite{M2}.
By Martin \cite[Theorem 4.1]{GM}, since $\beta_n$ is conjugate to $A_n$ which is of Type 1 with $d>0$ and some $a_i >0$, we have that $s(K_n)=w(K_n)-2$ where $w$ denotes the writhe of the knot.  
Recall that the writhe is the number of positive crossings minus the number of negative crossings in the knot diagram. 
Hence 
\[
w(K_n)=7-(2n+5)=-2n+2.
\]
We conclude that 
\[
s(K_n) = -2n-2+2 = -2n. 
\]
\end{proof}

\subsection{The $\tau$ invariant of $K_n$}
We will show that the $\tau$-defect $\delta_\tau$ vanishes for each knot $K_n$. 
\begin{prop} \label{prop:tau}
For $n = 1, 2, \dots$, the $\tau$ invariant of $K_n$ is $\tau(K_n) = -n$. 
\end{prop}
\begin{proof}
First, we may change a positive crossing in $K_n$ to a negative crossing to get a knot $P_n$.  Figure \ref{fig:crossingchange} shows the crossing change for $n=1$. After doing a Reidemeister I isotopy, and two Reidemeister II moves, we see that the knot $P_n$ is the $(2, -(2n+1))$--torus knot $T_{2, -(2n+1)}$. This sequence of isotopies is illustrated in Figure \ref{fig:tau}. Recall that $\tau$ invariant satisfies the crossing change inequality \cite[Corollary 1.5]{OS}
\[ 0 \leq \tau(K_n) - \tau(T_{2, -(2n+1)}) \leq 1. \]
Since  $\tau(T_{2, -(2n+1)}) = -n$, we have $ -n \leq \tau(K_n) \leq -n + 1$. 

Next, we may change a negative crossing in $K_n$ to a positive crossing to get a knot $R_n$. 
\[\tau(K_n) \leq \tau(R_n)\] 
by the crossing change inequality. We may then change a positive crossing in $R_n$ to a negative crossing to obtain the torus knot $T_{2, -(2n+3)}$. This processs is illustrated in Figure~\ref{fig:taustep2}. We have 
\[ \tau(R_n) \leq \tau(T_{2, -(2n+3)}) + 1 = -n.\]
Thus, we have $\tau(K_n) \leq -n$. Together with the first step, we find that $\tau(K_n) = -n$ for each positive integer $n$. 
\end{proof}

\begin{figure}
\includegraphics[scale=.3]{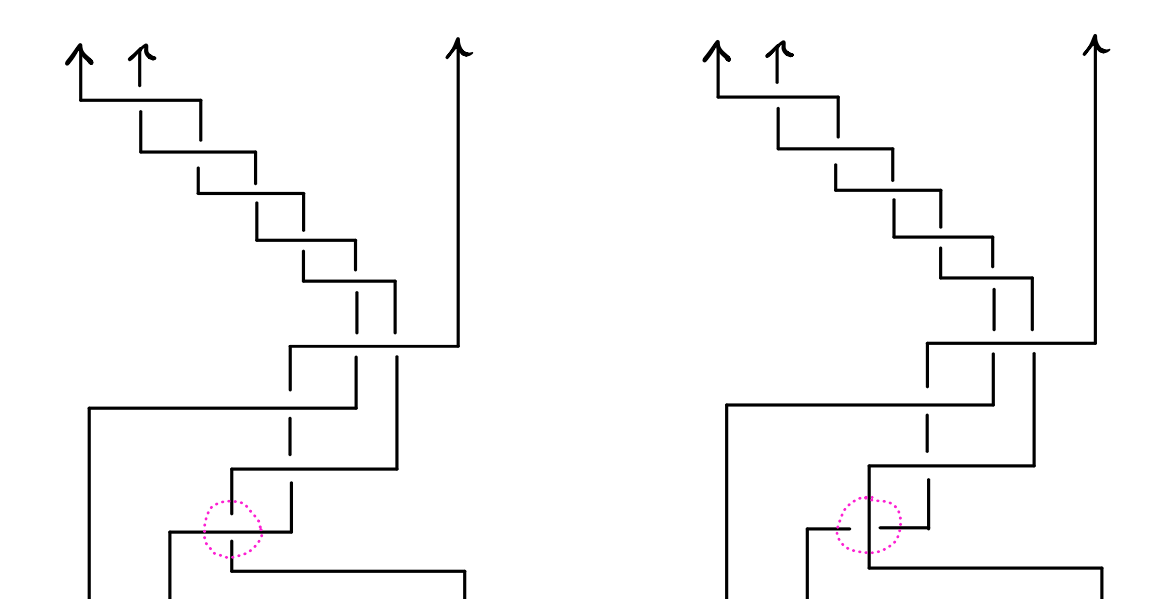}
\put(-250, -25){\fontsize{12}{11}$K_1$}
\put(-75, -25){\fontsize{12}{11}$P_1$}
\caption{The knot $K_1$ is the braid closure shown on the left. After a crossing change, we obtain the knot $P_1$ as the braid closure shown on the right.   }
\label{fig:crossingchange}
\end{figure}

\begin{figure}
\includegraphics[scale=.2]{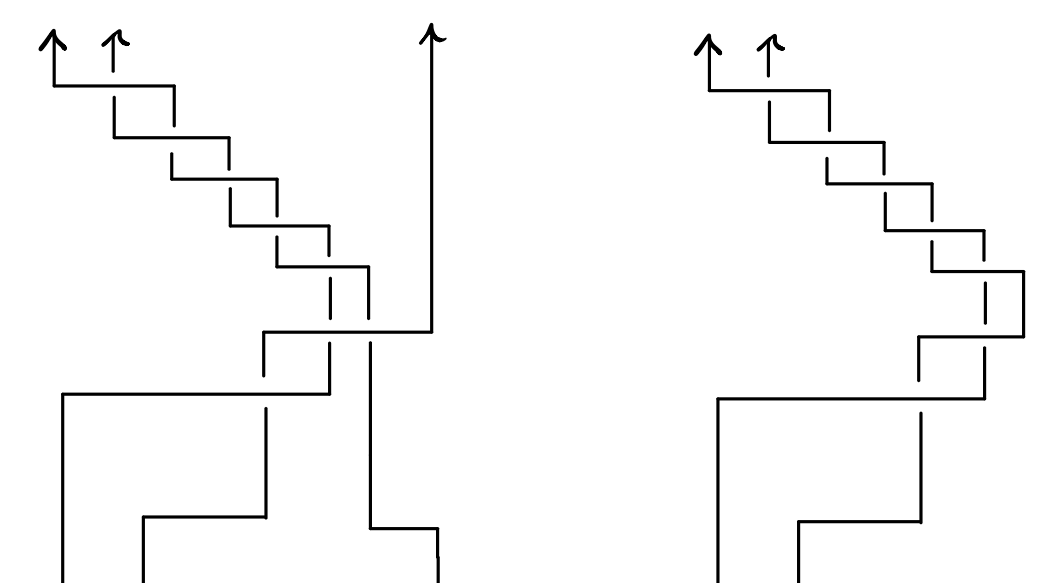} 
\quad \quad \quad
\includegraphics[scale=.2]{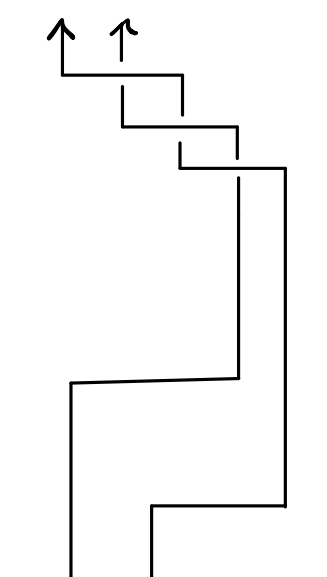}
\put(-280, -15){\fontsize{10}{11}$P_1$}
\put(-55, -15){\fontsize{10}{11}$T_{2, -3}$}
\put(-80, 55){\fontsize{10}{11}$\sim$}
\put(-210, 55){\fontsize{10}{11}$\sim$}
\caption{The knot $P_1$ is isotopic to the torus knot $T_{2, -3}$. The leftmost picture shows the knot $P_1$ after a Reidemeister II move. Perform a Reidemeister I move to obtain the knot in the center picture. Finally, perform two Reidemeister II moves to obtain $T_{2, -3}$ shown in the rightmost picture. }
\label{fig:tau}
\end{figure}

\begin{figure}
\includegraphics[scale=.2]{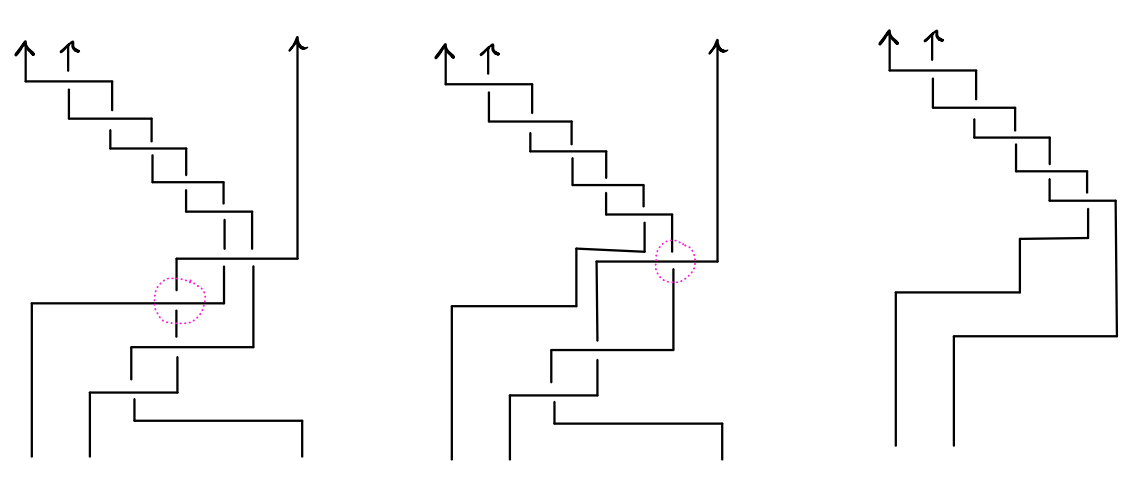} 
\put(-200, -15){\fontsize{10}{11}$P_1$}
\put(-120, -15){\fontsize{10}{11}$R_1$}
\put(-40, -15){\fontsize{10}{11}$T_{2,-5}$}
\caption{The knot $P_n$ is a single crossing change away from the knot $R_n$. The knot $R_n$ is a single crossing change away from the torus knot $T_{2,-(2n+3)}$. The illustrations are shown for $n=1$. Note that the crossing changes and Reidemeister moves occur away from the twisting region specified by $n$.}
\label{fig:taustep2}
\end{figure}

\subsection{The transverse and contact invariants of $K_n$}\label{section:thetapsi}

This section is dedicated to exploring invariants in the literature that can be used to detect if a knot is nonquasipositive. 
We study the Ozsv\'ath-Szab\'o-Thurston transverse invariant $\hat{\theta}(K)$ from knot Floer homology \cite{OST} and Plamenevskaya's transverse invariant $\psi(K)$ from Khovanov homology \cite{P2}. 
Recall that for quasipositive knots, the transverse invariants $\psi(K)$ and $\hat{\theta}(K)$ are both nonzero by \cite{P1}. 
Each knot $K_n$ is nonquasipositive by Corollary \ref{cor1}. 
However, the propositions below show that the nonquasipositive property of the knots $K_n$ is not detected by $\hat{\theta}(K)$ and $\psi(K)$.

\begin{prop}
For $n=1,2,\dots$, $\psi(K_n) \neq 0$.
\end{prop}

\begin{proof}
In \cite[Proposition 2.10]{GM}, Martin proved that for any $n$-braid $\beta$, if $s(\hat \beta)-1=w(\beta) - n$ then $\psi(\hat\beta)\neq 0$. 
In the proof of Proposition \ref{prop3}, we discovered that 
$s(K_n)=w(K_n)-2,
$
which satisfies Martin's condition.
Therefore, $\psi(K_n) \neq 0$. 
\end{proof}

\begin{prop}
For $n=1,2,\dots$, $\hat{\theta}(K_n) \neq 0$. 
\end{prop}

\begin{proof}
By Proposition \ref{prop3}, we know that $sl(K_n) = s(K_n)-1$.
Plamenevskaya's \cite[Proposition 3.2]{P1} shows that $K_n$ is right-veering for all $n$. 
Furthermore, by \cite[Theorem 1.2]{P1}, $\hat{\theta}(K_n) \neq 0$ for all $n$.
\end{proof}

\begin{cor}
For $n=1,2,\dots$, the Heegaard Floer contact invariant of $K_n$ does not vanish, or $c(\xi_{K_n}) \neq 0$. 
\end{cor}

\begin{proof}
By Corollary 4.2 of \cite{P1}, since $\hat{\theta}(K_n) \neq 0$, the Heegaard Floer contact invariant $c(\xi_{K_n}) \neq 0$. 
\end{proof}


\bibliographystyle{alpha}

\bibliography{references}

\end{document}